\documentclass{article}
\usepackage [utf8] {inputenc}
\usepackage{mathtools}
\usepackage{amsthm}
\usepackage{amssymb}
\usepackage{hyperref}

\newtheorem{theorem}{Theorem}[section]
\newtheorem{proposition}[theorem]{Proposition}
\newtheorem{lemma}[theorem]{Lemma}
\newtheorem{corollary}[theorem]{Corollary}

\DeclareMathOperator{\Sym}{Sym}
\DeclareMathOperator{\End}{End}
\DeclareMathOperator{\Aut}{Aut}
\DeclareMathOperator{\GL}{GL}

\title{Polynomial-time isomorphism test for $k$-generated extensions of abelian groups}
\author{Saveliy V. Skresanov}
\date{}

\begin{document}
\maketitle

\begin{abstract}
	The group isomorphism problem asks whether two finite groups given by their Cayley tables
	are isomorphic or not. Although there are poly\-nomial-time algorithms for some specific group classes,
	the best known algorithm for testing isomorphism of arbitrary groups of order \( n \) has time complexity \( n^{O(\log n)} \).
	We consider the group isomorphism problem for some extensions of abelian groups by \( k \)-generated groups for bounded \( k \).
	In particular, we prove that one can test isomorphism of abelian-by-cyclic extensions in polynomial time,
	generalizing a 2009 result of Le Gall for coprime extensions. As another application, we give a polynomial-time
	isomorphism test for abelian-by-simple group extensions, generalizing a 2017 result of Grochow and Qiao for central extensions.
	The main novelty of the proof is a polynomial-time algorithm for computing the unit group of a finite ring, which might be of independent interest.
\end{abstract}

\section{Introduction}

The \emph{group isomorphism problem} in complexity theory is a decision problem which asks if two finite groups
given by their Cayley tables are isomorphic or not. Since a group of order \( n \) has at most \( \log n \) generators
(all logarithms are base~\( 2 \) in this work), there is a trivial algorithm which tests isomorphism (and in fact, enumerates all isomorphisms)
between two groups of order \( n \) in time \( n^{O(\log n)} \), see~\cite{miller}.
The group isomorphism problem is polynomial-time reducible to the graph isomorphism problem, so since Babai's 2015 quasipolynomial
algorithm for the latter task~\cite{babai}, group isomorphism stands as a natural obstacle to further reducing complexity of graph isomorphism.

There have been many improvements for the group isomorphism problem for particular group classes, see, for example, the overview in~\cite{grochow}.
Babai, Codenotti and Qiao showed~\cite{trivrad} that isomorphism of groups with no normal abelian subgroups can be tested in polynomial time,
in particular, their result highlights the importance of studying the group isomorphism problem for extensions of abelian groups.
To fix our terminology, we say that a group \( G \) with a normal subgroup \( A \) is an \emph{extension} of \( A \) by \( G/A \);
we say that an extension is \emph{coprime}, if \( |A| \) is coprime to \( |G/A| \).
Le Gall proved~\cite{legall} that the group isomorphism problem is solvable in polynomial time for coprime abelian-by-cyclic extensions,
Qiao, Sarma and Tang~\cite{qiao} generalized that result to coprime extensions of abelian groups by \( k \)-generated groups for bounded \( k \),
where time complexity of the algorithm depended on \( k \) exponentially. The most general result of this flavor is by Babai and Qiao~\cite{absyl}
which gives a polynomial isomorphism test for groups with abelian Sylow towers of normal subgroups; the key step in their algorithm
resolves the isomorphism problem for certain coprime extensions of abelian groups~\cite[Theorem~1.2]{absyl}.

One of the reasons why there has been so much progress for coprime extensions is the fact that such extensions are \emph{split}:
if \( G \) is a coprime extension of \( A \) by \( G/A \), then by Schur--Zassenhaus theorem there exists a subgroup \( H \leq G \), \( H \simeq G/A \),
such that \( G \) is a semidirect product \( A \rtimes H \). Moreover, when \( A \) is abelian, one can often utilize results from representation theory
to study the action of \( H \) on~\( A \). For example, by~\cite[Theorem~1.2]{absyl}, if \( A \) is abelian and the full automorphism group of \( H \) is known,
then the full automorphism group of the coprime extension \( A \rtimes H \) can be computed in polynomial time, which essentially reduces
the problem for the original group \( G \) to that problem for~\( H \).

It is, therefore, important to develop similar results for noncoprime extensions of abelian groups.
In general, noncoprime extensions can be nonsplit, i.e.\ they are not necessarily isomorphic to a semidirect product of some subgroups of the group.
Another difficulty comes from the fact that the isomorphism class of a noncoprime extension is determined not only by the action of the ``top group'' \( G/A \)
on the ``bottom group'' \( A \), but also by the so-called \( 2 \)-cohomology class of the extension. Substantial progress in studying the group isomorphism problem
for extensions via cohomology was made by Grochow and Qiao in~\cite{grochow}, with the most strong results for central extensions, i.e.\ when \( A \) is a subgroup
of the center. For example, it follows from~\cite[Corollary~6.3]{grochow} that the group isomorphism problem can be solved in polynomial time
for central extensions of abelian groups by nonabelian simple groups; if the bottom group \( A \) is elementary abelian, then one can also
compute the coset of isomorphisms in polynomial time. In~\cite[Theorem~6.1]{grochow} the authors obtain a general isomorphism test for central extensions
whose complexity depends on the time required to enumerate the full automorphism group of \( G/A \).

We note that despite the progress in~\cite{grochow}, still not much is known for noncoprime extensions in the case when \( A \) is not central
or elementary abelian. In this work we focus on extensions of abelian groups (not necessarily central or elementary abelian)
by \( k \)-generated groups for bounded \( k \). First, we show that if \( G \) and \( G_0 \) are two extensions
of abelian normal subgroups \( A \unlhd G \) and \( A_0 \unlhd G_0 \), then one can check in polynomial time whether an isomorphism
between top groups \( G/A \) and \( G_0/A_0 \) extends to the isomorphism between \( G \) and \( G_0 \).

\begin{theorem}\label{main}
	Let \( G \) and \( G_0 \) be finite groups of order \( n \) given by their Cayley tables.
	Suppose we are also given abelian normal subgroups \( A \unlhd G \) and \( A_0 \unlhd G_0 \),
	and an isomorphism \( \psi : G/A \to G_0/A_0 \). If \( G/A \) is \( k \)-generated, then
	we can test in time polynomial in \( n^k \) whether there exists an isomorphism \( \Phi : G \to G_0 \) such that \( \Phi(A) = A_0 \)
	and \( \Phi \) induces \( \psi \) between \( G/A \) and \( G_0/A_0 \).
	If such an isomorphism exists, then we can compute the coset of such isomorphisms in the same time.
\end{theorem}

If we have two isomorphisms \( \Phi_1, \Phi_2 : G \to G_0 \) as in Theorem~\ref{main}, then
\( \Phi_1 \Phi_2^{-1} : G \to G \) is an automorphism of \( G \) which stabilizes \( A \) and acts trivially on \( G/A \).
In particular, the set of all isomorphisms \( \Phi : G \to G_0 \) inducing \( \psi \) between \( G/A \) and \( G_0/A_0 \)
can be represented by a coset \( \mathrm{Aut}_0(G, A) \cdot \Phi_1 \),
where \( \mathrm{Aut}_0(G, A) \) is a subgroup of \( \Aut(G) \) consisting of all elements which stabilize \( A \) and act trivially on \( G/A \).

We note that if \( k \) generators of \( G/A \) are given explicitly as a part of the input in Theorem~\ref{main},
then it follows from the proof that we can find an isomorphism (and the coset of isomorphisms) in time polynomial in \( n \) and \( |A|^k \).

In Theorem~\ref{main} the subgroups \( A, A_0 \) and isomorphism \( \psi : G/A \to G_0/A_0 \) are given as a part of the input.
In order to turn Theorem~\ref{main} into a proper isomorphism test, we need to be able to find suitable \( A, A_0 \) and \( \psi \) in polynomial time.
The following result covers an important special case when this is possible. The key property is that if \( G/A \) is \( k \)-generated,
then all isomorphisms from \( G/A \) to \( G_0/A_0 \) can be listed in time polynomial in \( n^k \).

\begin{theorem}\label{ctower}
	Let \( G \) and \( G_0 \) be finite groups of order \( n \) given by their Cayley tables.
	Suppose that \( G \) has a normal abelian subgroup \( A \) such that \( G/A \) has a subnormal series
	of length \( k \) with factors which are cyclic or simple groups.
	Then we can test in time polynomial in \( n^k \) whether \( G \) and \( G_0 \) are isomorphic, and find the coset of isomorphisms, if they are.
\end{theorem}

The normal subgroup \( A \) is not given as a part of the input in Theorem~\ref{ctower}, but we can find it in polynomial time, see Corollary~\ref{findall}.
Note also that the coset of isomorphisms in Theorem~\ref{ctower} is a \( \Aut(G) \)-coset, in particular, in the assumptions of Theorem~\ref{ctower}
we can compute \( \Aut(G) \) in polynomial time.

To illustrate the applications of Theorem~\ref{ctower}, we make two simple corollaries.

\begin{corollary}
	Let \( G \) and \( G_0 \) be finite groups of order \( n \) given by their Cayley tables.
	Suppose that \( G \) is an extension of an abelian group by a \( k \)-generated abelian group.
	Then we can test in time polynomial in \( n^k \) whether \( G \) and \( G_0 \) are isomorphic, and find the coset of isomorphisms, if they are.
\end{corollary}

In particular, when \( k = 1 \) the above corollary resolves the isomorphism problem for arbitrary abelian-by-cyclic group extensions.
This has been known for coprime extensions by~\cite{legall}, but as far as the author is aware, the case of arbitrary extensions was open.

\begin{corollary}
	Let \( G \) and \( G_0 \) be finite groups of order \( n \) given by their Cayley tables.
	Suppose that \( G \) is an extension of an abelian group by a direct product of \( k \) simple groups.
	Then we can test in time polynomial in \( n^k \) whether \( G \) and \( G_0 \) are isomorphic, and find the coset of isomorphisms, if they are.
\end{corollary}

This corollary should be compared to~\cite[Theorem~7.1]{grochow} which applies only to the case when the bottom group is central and elementary abelian
(though we note that~\cite[Theorem~7.1]{grochow} does not make any assumptions on the number \( k \) of direct factors of the top group).

Note that one can obtain applications of Theorem~\ref{main} beyond Theorem~\ref{ctower}. For example, since the solvable radical of a finite group
can be computed in polynomial time, see~\cite[Section~6.3.1]{seress}, one can obtain the following corollary of Theorem~\ref{main}.

\begin{corollary}
	Let \( G \) and \( G_0 \) be finite groups of order \( n \) given by their Cayley tables.
	Suppose that the solvable radical \( A \) of \( G \) is abelian, and \( G/A \) is \( k \)-generated.
	Then we can test in time polynomial in \( n^k \) whether \( G \) and \( G_0 \) are isomorphic, and find the coset of isomorphisms, if they are.
\end{corollary}

Since \( k \)-generated groups with trivial solvable radical can have arbitrarily long subnormal series with simple factors, the above corollary does not follow
from Theorem~\ref{ctower}.

We note that it is an interesting question whether one can get rid of any assumptions on the number of generators \( k \) in Theorem~\ref{main},
though such a result seems to require cohomological techniques in the spirit of~\cite{grochow}.
If that was possible, that would allow one, for example, to build a polynomial isomorphism test for groups \( G \) of order \( n \) with abelian solvable radical \( A \),
such that \( G/A \) has at most \( O(\log n / \log \log n) \) minimal normal subgroups, cf.~\cite[Corollary~6.13]{grochow}.

The main novelty which allowed us to prove Theorem~\ref{main} is the following result on the computation of unit groups of finite rings,
which might be of independent interest. The way rings are encoded in our algorithms is explained in detail in Section~\ref{sprem}.

\begin{theorem}\label{units}
	Let \( R \) be a finite unitary ring given by its additive generators.
	If \( p_{\max} \) is the largest prime divisor of \( |R| \),
	then we can compute the multiplicative generators of \( R^\times \) in time polynomial in \( \log |R| \) and \( p_{\max} \).
\end{theorem}

Our proof of Theorem~\ref{units} follows the proof of a result by Hofmann~\cite{hofmann}. It was shown in~\cite[Theorem~A]{hofmann}
that the problem of computing an effective presentation of the unit group of a finite ring is probabilistic polynomial-time equivalent to the factoring
and discrete logarithm problems, for which no polynomial-time algorithm is known.
It turns out that this difficulty can be overcome if the time complexity of the algorithm is allowed to depend on the
size of the largest prime divisor of the order of the ring. In our application to the group isomorphism problem
\( R \) will be a subring of the endomorphism ring of a finite abelian group \( A \) given by its Cayley table, see Corollary~\ref{cunits}.
In particular, all prime divisors of \( |R| \) will be bounded by \( |A| \).

The structure of the paper is as follows. In Section~\ref{sprem} we give computational preliminaries, putting
special emphasis on the different ways objects like groups, rings and modules are represented in our algorithms.
In Section~\ref{siso} we give an explicit criterion on when two group extensions are isomorphic, which is later used in the algorithms.
Section~\ref{sunits} is devoted to the proof of Theorem~\ref{units}, while Sections~\ref{sproof} and~\ref{stower} prove Theorems~\ref{main} and~\ref{ctower}, respectively.

\section{Preliminaries}\label{sprem}

Recall that an algorithm works in \emph{polynomial time}, if it works in time which is bounded by a polynomial in the size of the input.
Our algorithms will mostly work with finite groups, rings and modules and it is, therefore, important to specify how exactly these objects
are represented in the algorithms.

Let \( G \) be a finite group of order \( n \). We say that \( G \) is given by its \emph{Cayley table} to an algorithm,
if \( G \) is encoded by its \( n \times n \) multiplication table. We may assume that all elements of \( G \)
are numbered from \( 1 \) to \( n \), so the Cayley table has size \( O(n^2 \log n) \) if we encode each table cell in binary.
Therefore in this case an algorithm works in polynomial time if it works in time polynomial in \( n \).

Subgroups of \( G \) can be specified as a list of their elements, homomorphisms (and isomorphisms) between two groups given by their Cayley tables
are specified by a list of element-image pairs. Many standard operations for groups can be easily performed in polynomial time for groups given by their Cayley tables,
for example: given a subset of \( G \) one can find the subgroup it generates, given a normal subgroup one can compute the Cayley table of
the quotient group and the associated natural homomorphism, if \( G \) is \( k \)-generated, one can list all \( k \)-tuples which generate \( G \)
in time polynomial in \( n^k \).

Let \( G \leq \Sym(m) \) be a permutation group of degree \( m \). We say that \( G \) is given by \emph{permutation generators},
if we are given a subset \( S \) of \( G \) which generates \( G \). Note that any permutation from \( \Sym(m) \)
can be specified by a list of \( m \) elements, so it occupies \( O(m \log m) \) space. 
Hence, if \( G \) is given to an algorithm by permutation generators, then the algorithm works in polynomial time
if it works in time polynomial in \( |S|\cdot m \cdot \log m \). It is well-known, see for instance~\cite[Exercise~4.1]{seress},
that any generating set \( S \) can be reduced to at most \( m^2 \) generators in time polynomial in \( m \) and \( |S| \). Therefore
we may safely assume that an algorithm on permutation groups works in polynomial time if it works in time polynomial in the degree of the permutation group.
Subgroups of permutation groups are also specified by their generators.
We refer the reader to~\cite[Section~3.1]{seress} for the library of polynomial-time algorithms for permutation groups, note, in particular,
that tasks like finding orbits and pointwise stabilizers can be performed in polynomial time.

Let \( A \) be a finite abelian group. It is well-known that \( A \) is isomorphic to a
direct sum of \( t \) cyclic groups of orders \( n_1, \dots, n_t \).
We say that \( A \) is given by its \emph{cyclic generators}, if the algorithm is given a \( t \)-tuple \( (n_1, \dots, n_t) \) of integers as input.
The length of the input is \( l = O(\sum_{i=1}^t \log n_i) \), and elements of \( A \) can be represented by \( t \)-tuples \( (x_1, \dots, x_t) \),
where \( x_i \in \mathbb{Z}_{n_i} \), \( i = 1, \dots, t \). If \( e_i = (0, \dots, 1, \dots, 0) \) is a generator of the \( i \)-th direct summand
\( \mathbb{Z}_{n_i} \), \( i = 1, \dots, t \), then clearly \( e_1, \dots, e_t \) generate \( A \). With some abuse of terminology, we
say that \( e_1, \dots, e_t \) are the cyclic generators of \( A \). Observe also that \( l = O(\log |A|) \).

Subgroups of \( A \) will be specified by their generating sets;
we note that any set of generators \( S \subseteq A \) of a subgroup can be reduced to a set of size at most \( t \) in time polynomial in \( |S| \) and~\( l \).
Homomorphisms from \( A \) into some other group are specified by listing images of \( e_1, \dots, e_t \) under the homomorphism in question.
We refer the reader to~\cite[Chapter~2]{iuliana-thesis} for a list of polynomial-time algorithms for abelian groups given by cyclic generators.
Observe also, that if the group \( A \) is given by its Cayley table, then in time polynomial in \( |A| \) one can compute a representation of \( A \) by cyclic generators
together with an explicit isomorphism from \( A \) to \( \mathbb{Z}_{n_1} \oplus \dots \oplus \mathbb{Z}_{n_t} \), see, for instance,~\cite{kayal}.

Let \( R \) be a finite (not necessarily commutative) unitary ring, 
and let \( R^+ \) denote its additive group. We say that \( R \) is given by \emph{additive generators},
if the abelian group \( R^+ \simeq \mathbb{Z}_{n_1} \oplus \dots \oplus \mathbb{Z}_{n_t} \) is given by its cyclic generators
and the ring multiplication \( \cdot : R \times R \to R \) is specified by \( t^3 \) numbers \( \alpha_{ijk} \) such that
\( e_i \cdot e_j = \sum_{k=1}^t \alpha_{ijk} e_k \), where \( e_1, \dots, e_t \) are the cyclic generators of \( R^+ \).
The length of the input in this representation can be bounded by \( l = O(t^2 \cdot \sum_{i=1}^t \log n_i) \).
Note that \( \log |R| \) is bounded polynomially in terms of \( l \) and vice versa, \( l = O((\log |R|)^3) \).
In particular, an algorithm working with a finite ring \( R \) works in polynomial time if and only if the algorithm works in time polynomial in \( \log |R| \).

With some abuse of terminology, we will say that a subring or an ideal \( I \) in \( R \) is given by additive generators,
if we are given a set of elements of \( R \) which spans \( I \) additively.
Homomorphisms between rings are represented as homomorphisms between underlying abelian groups.
We refer the reader to~\cite[Chapter~3]{iuliana-thesis} for some basic polynomial-time algorithms for finite rings.
In particular, we note that tasks like computing a quotient of a ring by its ideal, and computing sums and products of ideals
can be performed in polynomial time.

Let \( A \simeq \mathbb{Z}_{n_1} \oplus \dots \oplus \mathbb{Z}_{n_t} \)
be an abelian group given by its cyclic generators \( a_1, \dots, a_t \), and let \( R \) be a finite unitary ring
given by its additive generators \( e_1, \dots, e_m \). If \( A \) has a structure of a right \( R \)-module, then in order
to give the module action of \( R \) on \( A \) to an algorithm, we specify \( t^2 \cdot m \) numbers \( \beta_{ijk} \)
such that \( a_i^{e_j} = \sum_{k=1}^t \beta_{ijk} a_k \); here \( a^x \) denotes the image of \( a \in A \) under the action of \( x \in R \).
The length of the input in this case is bounded by \( O(tm \cdot \sum_{i=1}^t \log n_i) \).
Homomorphisms between modules are represented as homomorphisms between underlying groups.
The following result shows that isomorphism of finite modules can be determined in polynomial time.

\begin{proposition}[{\cite[Theorem~1.1]{iuliana} and \cite[Proposition~3.1.7]{iuliana-thesis}}]\label{modiso}
	Let \( R \) be a finite unitary ring given by additive generators,
	and let \( A \) and \( A_0 \) be two finite \( R \)-modules given by cyclic generators.
	There exists a polynomial-time algorithm which decides whether \( A \) and \( A_0 \) are isomorphic as \( R \)-modules,
	and if they are, one can compute an explicit \( R \)-isomorphism \( \mu : A \to A_0 \) in polynomial time.
\end{proposition}

Let \( Ax = b \) be a system of linear Diophantine equations, where \( A \in M_{mn}(\mathbb{Z}) \), \( b \in \mathbb{Z}^m \)
and \( x \) is a vector of unknowns. Recall that the set of solutions to \( Ax = b \) is given by a particular solution \( y \in \mathbb{Z}^n \)
and a set of vectors \( v_1, \dots, v_t \in \mathbb{Z}^n \) which generate the null space of the homogeneous system \( Ax = 0 \).
We input a system \( Ax = b \) to an algorithm by specifying \( mn + m \) integers in binary, so the input length is bounded by
\( O(\sum_{i,j} \log a_{ij} + \sum_i \log b_i) \), where \( A = (a_{ij}) \) and \( b = (b_i) \).

\begin{proposition}[{\cite{dio1, dio2}}]\label{diosolve}
	A system of linear Diophantine equations can be solved in polynomial time.
\end{proposition}

\section{Isomorphisms of extensions}\label{siso}

Given a group word \( w(t_1, \dots, t_k) \) in variables \( t_1, \dots, t_k \), we write \( w(\overline{t}) \) as a shorthand.
Similarly, if \( g_1, \dots, g_k \) are elements of some group \( G \), we write \( w(\overline{g}) \) for \( w(g_1, \dots, g_k) \in G \).

The following result gives a criterion when two group extensions are isomorphic.

\begin{theorem}\label{kgenext}
	Let \( G \) and \( G_0 \) be finite groups, and let \( A \unlhd G \) and \( A_0 \unlhd G_0 \).
	Let a group \( H \) be given by generators and relators:
	\[ H = \langle x_1, \dots, x_k \mid w_i(\overline{x}) = 1,\, i = 1, \dots, m \rangle. \]
	Let \( g_1, \dots, g_k \in G \) be such that \( g_iA \), \( i = 1, \dots, k \), generate \( G/A \) and the map \( g_iA \mapsto x_i \)
	induces an isomorphism between \( G/A \) and \( H \). Similarly, let \( g^0_1, \dots, g^0_k \in G_0 \) be such that \( g^0_iA_0 \), \( i = 1, \dots, k \),
	generate \( G_0/A_0 \) and \( g^0_iA_0 \mapsto x_i \) induces an isomorphism between \( G_0/A_0 \) and \( H \).
	Let \( u_j = w_j(\overline{g}) \in A \), and \( u^0_j = w_j(\overline{g^0}) \in A_0 \), \( j = 1, \dots, m \).

	Suppose that there exists an isomorphism \( \phi : A \to A_0 \) such that \( \phi(x^{g_i}) = \phi(x)^{g^0_i} \) for all \( x \in A \) and \( i = 1, \dots, k \),
	and \( \phi(u_j) = u_j^0 \) for all \( j = 1, \dots, m \). Then there exists a unique isomorphism \( \Phi : G \to G_0 \) which extends \( \phi \)
	and maps \( g_i \) to \( g^0_i \) for \( i = 1, \dots, k \).
\end{theorem}

Note that the requirements of Theorem~\ref{kgenext} are easily satisfied if the groups \( G \) and \( G_0 \) are isomorphic.
Indeed, let \( \Phi : G \to G_0 \) be some isomorphism, and let \( A \unlhd G \) and \( g_1, \dots, g_k \in G \) be such that
\( g_iA \), \( i = 1, \dots, k \) generate \( G/A \). Set \( H = G/A \), and let \( x_i = g_iA \), \( i = 1, \dots, k \) be its generators.
Let \( w_j(\overline{x}) \), \( j = 1, \dots, m \) be some relators and \( u_j = w_j(\overline{g}) \), \( j = 1, \dots, m \).
Now if we set \( A_0 = \Phi(A) \), \( g^0_i = \Phi(g_i) \), \( i = 1, \dots, k \), and \( u^0_j = \Phi(u_j) \), \( j = 1, \dots, m \),
and let \( \phi \) be the restriction of \( \Phi \) on \( A \), then this choice of subgroups and elements satisfies the hypotheses of Theorem~\ref{kgenext}.
\medskip

\noindent
\emph{Proof of Theorem~\ref{kgenext}.}
Given some \( x \in G \), there exists a word \( w(t_1, \dots, t_k) \) such that \( x = w(g_1, \dots, g_k)a = w(\overline{g})a \) for some \( a \in A \).
We define
\[ \Phi(x) = w(g^0_1, \dots, g^0_k)\phi(a) = w(\overline{g^0})\phi(a) \in G_0. \]
Let us check that this definition is correct and does not depend on the choice of \( w \) and \( a \in A \).

Suppose that \( x = w'(\overline{g})a' \) for some word \( w'(t_1, \dots, t_k) \) and \( a' \in A \).
Let \( \psi : G \to H \) be a homomorphism with kernel \( A \) such that \( \psi(g_i) = x_i \), \( i = 1, \dots, k \);
it exists since \( g_iA \mapsto x_i \) extends to an isomorphism between \( G/A \) and \( H \). Now,
\[ \psi(x) = w(x_1, \dots, x_k) = w'(x_1, \dots, x_k). \]
Since \( w(\overline{x}) \) and \( w'(\overline{x}) \) are equal in \( H \), the word \( w(\overline{t})^{-1}w'(\overline{t}) \)
lies in the subgroup of the free group \( F = \mathrm{Free}(t_1, \dots, t_k) \) generated by \( w_j(\overline{t})^z \), \( z \in F \), \( j = 1, \dots, m \).
Without loss of generality, we may assume that the list of relators \( w_1, \dots, w_m \) contains all inverses, i.e.\ for any \( w_j(\overline{t}) \)
there exists \( j' \in \{ 1, \dots, m \} \) such that \( w_j(\overline{t})^{-1} = w_{j'}(\overline{t}) \). Indeed, the addition of inverse
relators does not impose any new restrictions on the isomorphism \( \phi \), since \( \phi(u_j) = u_j^0 \) implies \( \phi(u_j^{-1}) = (u_j^0)^{-1} \).
In particular, \( w(\overline{t})^{-1}w'(\overline{t}) \) can be expressed as a product of elements of the form
\( w_j(\overline{t})^z \), \( z \in F \), \( j = 1, \dots, m \), and we can write
\begin{equation}\label{eq1}
	w'(\overline{t}) = w(\overline{t}) \cdot w_{j_1}(\overline{t})^{z_1(\overline{t})} \cdots w_{j_n}(\overline{t})^{z_n(\overline{t})}
\end{equation}
for some \( j_1, \dots, j_n \in \{ 1, \dots, m \} \) and \( z_1, \dots, z_n \in F \).

In order to prove that \( \Phi(x) \) is defined correctly, we need to show that \( w(\overline{g^0})\phi(a) = w'(\overline{g^0})\phi(a') \).
If we substitute \( g_1, \dots, g_k \) into Equation~(\ref{eq1}) and multiply the equality by \( a' \) from the right, we get
\[ w'(\overline{g})a' = w(\overline{g}) \cdot w_{j_1}(\overline{g})^{z_1(\overline{g})} \cdots w_{j_n}(\overline{g})^{z_n(\overline{g})}a' =
	w(\overline{g}) \cdot u_{j_1}^{z_1(\overline{g})} \cdots u_{j_n}^{z_n(\overline{g})}a'.  \]
The left-hand side of that equality is equal to \( x = w(\overline{g})a \), hence
\[ w(\overline{g})a = w(\overline{g}) \cdot u_{j_1}^{z_1(\overline{g})} \cdots u_{j_n}^{z_n(\overline{g})}a' \]
which implies \( a = u_{j_1}^{z_1(\overline{g})} \cdots u_{j_n}^{z_n(\overline{g})}a' \). If we apply \( \phi \) to that equality, we get
\begin{equation}\label{eq2}
	\phi(a) = (u^0_{j_1})^{z_1(\overline{g^0})} \cdots (u^0_{j_n})^{z_n(\overline{g^0})}\phi(a').
\end{equation}
Here we are using the fact that if \( y \in A \) and \( z(g_1, \dots, g_k) \) is an arbitrary word in \( g_1, \dots, g_k \),
then \( \phi(y^{z(\overline{g})}) = \phi(y)^{z(\overline{g^0})} \).

If we substitute \( g^0_1, \dots, g^0_k \) into Equation~(\ref{eq1}), we obtain
\begin{align*}
	w'(\overline{g^0}) &= w(\overline{g^0}) \cdot w_{j_1}(\overline{g^0})^{z_1(\overline{g^0})} \cdots w_{j_n}(\overline{g^0})^{z_n(\overline{g^0})}\\
			&= w(\overline{g^0}) \cdot (u^0_{j_1})^{z_1(\overline{g^0})} \cdots (u^0_{j_n})^{z_n(\overline{g^0})}.
\end{align*}
If we multiply this equality by \( \phi(a') \) from the right, we get
\[ w'(\overline{g^0})\phi(a') = w(\overline{g^0}) \cdot (u^0_{j_1})^{z_1(\overline{g^0})} \cdots (u^0_{j_n})^{z_n(\overline{g^0})}\phi(a') = w(\overline{g^0})\phi(a), \]
where the second equality follows from Equation~(\ref{eq2}).
Therefore \( \Phi(x) \) does not depend on the choice of \( w \) and \( a \), and hence \( \Phi \) is defined correctly.

It follows immediately from the definition of \( \Phi \) that \( \Phi(a) = \phi(a) \) for \( a \in A \), and \( \Phi(g_i) = g^0_i \) for \( i = 1, \dots, k \).
It is left to show that \( \Phi \) is a homomorphism; note that it automatically will be an isomorphism since \( A_0 = \Phi(A) \) and \( g^0_i = \Phi(g_i) \),
\( 1, \dots, k \), generate \( G_0 \).

Take \( x, y \in G \), and let \( x = w(\overline{g})a \), \( y = w'(\overline{g})a' \) for some words \( w, w' \) and elements \( a, a' \in A \).
Then \( \Phi(x) = w(\overline{g^0})\phi(a) \) and \( \Phi(y) = w'(\overline{g^0})\phi(a') \). Therefore
\begin{multline*}
	\Phi(xy) = \Phi(w(\overline{g})a \cdot w'(\overline{g})a') = \Phi(w(\overline{g})w'(\overline{g}) \cdot a^{w'(\overline{g})}a') = \\
	= w(\overline{g^0})w'(\overline{g^0}) \phi(a^{w'(\overline{g})}a') = w(\overline{g^0})w'(\overline{g^0}) \phi(a)^{w'(\overline{g^0})}\phi(a') = \\
	= w(\overline{g^0})\phi(a) \cdot w'(\overline{g})\phi(a') = \Phi(x)\Phi(y)
\end{multline*}
as required. Finally, the uniqueness of \( \Phi \) easily follows from the fact that it coincides with \( \phi \) on \( A \)
and the condition \( \Phi(g_i) = g^0_i \), \( i = 1, \dots, k \). \qed

\section{Computing unit groups of finite rings}\label{sunits}

Let \( R \) be a finite unitary ring given by its additive generators, and let \( R^\times \) denote its unit group.
Our goal is to compute (polynomially many) elements of \( R \), which generate \( R^\times \) multiplicatively.

Recall that if \( J \) is a nilpotent ideal of \( R \), then every element of \( 1 + J = \{ 1 + x \mid x \in I \} \) is invertible,
so \( 1 + J \) is a subgroup of \( R^\times \).

\begin{lemma}\label{ideal}
	If \( R \) is a finite unitary ring given by additive generators
	and \( J \) is a nilpotent ideal of \( R \) given by additive generators \( T \), then we can compute multiplicative generators
	of \( 1 + J \) in time polynomial in \( \log |R| \) and \( |T| \).
\end{lemma}
\begin{proof}
	We follow the proofs of~\cite[Lemma~4.1 and Proposition~4.2]{hofmann}.
	Let \( I \) be some nilpotent ideal of \( R \) given by its additive generators. First we show how to compute elements \( 1+g_1, \dots, 1+g_r \in 1 + I \)
	such that \( (1+g_i)(1+I^2) \), \( i = 1, \dots, r \), generate \( (1+I)/(1+I^2) \) multiplicatively.

	Note that we can compute additive generators of \( I^2 \) in polynomial time by considering all pairwise products of the additive generators of~\( I \).
	Now, observe that the map \( \phi : 1+I \to I/I^2 \) given by \( \phi(1+x) = x+I^2 \), \( x \in I \), is a homomorphism
	from the multiplicative group \( 1+I \) to the additive group \( I/I^2 \). The kernel of \( \phi \) is \( 1+I^2 \),
	and note that for any \( g \in I/I^2 \) we can compute some preimage in \( 1+I \) in polynomial time.	
	We can compute at most \( \log |I/I^2| \leq \log |R| \) additive generators of \( I/I^2 \) in polynomial time,
	and let \( 1 + g_i \), \( i = 1, \dots, r \), be preimages in \( 1+I \) of these generators.

	Returning to the original task, let \( n \) be the minimal positive integer such that \( J^n = 0 \).
	Set \( k = \lceil \log(n) \rceil \) and \( J_i = J^{2^{k-i}} \), \( i = 0, \dots, k \).
	We have a series of normal subgroups \( N_i = 1 + J_i \), \( i = 0, \dots, k \), where \( N_0 = 1 \) and \( N_k = 1 + J \).
	By the argument from the previous paragraph applied to each quotient \( N_{i+1}/N_i \), \( i = 0, \dots, k-1 \),
	we obtain generators of the group \( N_k \); note that the total number of generators obtained is at most \( k \cdot \log |R| \).
	Since \( k \leq 1 + \log |R| \), the algorithm works in polynomial time.
\end{proof}

Recall that a finite ring \( R \) is \emph{semisimple}, if its Jacobson radical is trivial. If \( R \) is a \( \mathbb{Z}_p \)-algebra for some prime \( p \),
then the Wedderburn--Artin theorem implies that \( R \) is a direct sum of its minimal nontrivial ideals \( M_1, \dots, M_k \),
and every \( M_i \), \( i = 1, \dots, k \), is isomorphic to a full matrix ring over a finite field of characteristic \( p \).
We will be interested in finding these ideals and explicit isomorphisms with matrix rings computationally.

Let \( \mathbb{F} \) be a finite field of characteristic \( p \) and order \( p^f \), \( f \geq 1 \).
Since \( \mathbb{F} \) is a ring, we can specify it to our algorithms by additive generators. In particular, field operations
in \( \mathbb{F} \) can be performed in time polynomial in \( f \) and \( \log p \).
An \( n \times n \) matrix over \( \mathbb{F} \) is encoded directly in the algorithm (as a matrix of elements from~\( \mathbb{F} \))
and hence has size polynomial in \( n, f \) and \( \log p \).
We say that an isomorphism between some \( M_i \) and a matrix ring \( M_n(\mathbb{F}) \) can be found explicitly,
if given an element of \( M_i \) we can compute the corresponding \( n \times n \) matrix in polynomial time, and vice versa.

The following result is implicit in~\cite{ronyai}.

\begin{lemma}\label{wedder}
	Let \( R \) be a finite unitary semisimple ring given by its additive generators.
	Assume that \( R \) is a \( \mathbb{Z}_p \)-algebra for a prime \( p \), and \( d \) is the dimension of \( R \) over \( \mathbb{Z}_p \).
	The following tasks can be solved in time polynomial in \( p \) and \( d \).
	\begin{enumerate}
		\item Find minimal ideals \( M_1, \dots, M_k \) of \( R \).
		\item For each \( i = 1, \dots, k \), one can find an explicit isomorphism of \( M_i \) with a matrix ring \( M_{n_i}(\mathbb{F}_{q_i}) \)
			over a finite field \( \mathbb{F}_{q_i} \) of order \( q_i \) and characteristic~\( p \).
	\end{enumerate}
\end{lemma}
\begin{proof}
	Minimal ideals can be computed in polynomial time by~\cite[Corollary~3.2]{ronyai}.
	A polynomial-time algorithm for finding an explicit isomorphism from \( M_i \) to \( M_{n_i}(\mathbb{F}_{q_i}) \)
	follows from~\cite[Theorem~5.2]{ronyai}, see also~\cite[Corollary~4.7]{ronyai}. The algorithm in~\cite[Theorem~5.2]{ronyai} is an \( f \)-algorithm,
	i.e.\ it is a polynomial-time algorithm which has access to an oracle for factoring polynomials over finite fields.
	Since our algorithms are allowed to run in time polynomial in \( d \) and \( p \) (versus \( \log p \)),
	we can use Berlekamp's factoring algorithm~\cite{berle} to eliminate the use of the oracle.
\end{proof}

Now we are ready to prove the main result of this section and show that the unit group of a finite ring can be computed in polynomial time
if the largest prime divisor of the order of the ring is considered part of the input.
\medskip

\noindent
\emph{Proof of Theorem~\ref{units}.}
Let \( R \) be a finite unitary ring given by its additive generators, and let \( p_{\max} \) be the largest prime divisor of \( |R| \).
In time polynomial in \( p_{\max} \) and \( \log |R| \) we can compute \( |R| \) and decompose it into a product of prime power factors.
Given this factorization, decompose \( R^+ \) into a direct sum of abelian groups
\[ R^+ = R_1 \oplus \dots \oplus R_s, \]
where each \( R_i \) has prime power order, \( s \leq \log |R| \), and \( |R_i| \) and \( |R_j| \) are coprime for \( i \neq j \).
Observe that \( R_i \), \( i = 1, \dots, s \), is a subring of \( R \), and note that
\[ R^\times = R_1^\times \times \dots \times R_s^\times. \]
So it suffices to compute the unit group of a ring of prime power order.

Without loss of generality we may assume that \( |R| \) is a power of a prime~\( p \).
Let \( J \) be the Jacobson radical of \( R \). There is an exact sequence of groups
\[ 1 \longrightarrow 1+J \longrightarrow R^\times \longrightarrow (R/J)^\times \longrightarrow 1. \]
We can compute the ideal \( pR \) and the quotient-ring \( R/pR \) in polynomial time,
and since \( pR \leq J \), we can compute \( J \) by taking the preimage of the Jacobson radical of \( R/pR \).
The ring \( R/pR \) is a \( \mathbb{Z}_p \)-algebra, and its Jacobson radical can be computed in polynomial time by~\cite[Theorem~1.5A]{fr}.
The generators of \( 1 + J \) are obtained by Lemma~\ref{ideal}.

We can compute the quotient-ring \( R/J \) in polynomial time, and since \( R/J \) is a semisimple \( \mathbb{Z}_p \)-algebra,
it is a direct sum of its minimal ideals
\[ R/J = M_1 \oplus \dots \oplus M_k, \]
where each \( M_i \), \( i = 1, \dots, k \), is isomorphic to a full matrix ring \( M_{n_i}(q_i) \), and \( k \leq \log |R| \).
This decomposition and explicit isomorphisms between \( M_i \) and \( M_{n_i}(q_i) \) can be found in polynomial time by Lemma~\ref{wedder}.
Since
\[ (R/J)^\times = M_1^\times \times \dots \times M_k^\times, \]
and \( M_i^\times \simeq \GL_{n_i}(q_i) \), \( i = 1, \dots, k \),
it suffices to find multiplicative generators of \( \GL_{n_i}(q_i) \leq M_{n_i}(q_i) \).
Indeed, once we compute the multiplicative generators of \( M_i^\times \), \( i = 1, \dots, k \),
we can lift them back to \( R \) and together with generators of \( 1 + J \) we obtain the generators of \( R^\times \), as required.

Note that we can find generators of \( \mathrm{SL}_{n_i}(q_i) \) in polynomial time by~\cite[Section~3]{roncentr}.
In order to find generators of the subgroup of diagonal matrices
\[ D_{n_i}(q_i) = \{ \mathrm{diag}(\alpha_1, \dots, \alpha_{n_i}) \mid \alpha_j \in \mathbb{F}_{q_i}^\times,\, j = 1, \dots, n_i \} \]
it suffices to find generators of the multiplicative group of a finite field \( \mathbb{F}_{q_i} \).
By~\cite[Section~2]{lenstra}, in polynomial time we can find an irreducible polynomial \( f \in \mathbb{Z}_p[x] \)
and an explicit isomorphism between \( \mathbb{F}_{q_i} \) and the quotient-ring \( \mathbb{Z}_p[x] / (f) \).
By~\cite[Theorem~1.1]{shoup}, in time polynomial in \( \deg f = \log q_i / \log p \) and \( p \) we can find
a subset of elements of \( \mathbb{Z}_p[x] / (f) \) (of polynomial size) which contains a primitive root,
in particular, we can compute multiplicative generators of \( \mathbb{F}_{q_i}^\times \), and hence of \( D_{n_i}(q_i) \), in polynomial time.
Note that \( \deg f \leq \log |R| \) and \( p \leq p_{\max} \).

To finish the proof, we note that the generators of \( \GL_{n_i}(q_i) \) can be obtained as a union of generators of
\( \mathrm{SL}_{n_i}(q_i) \) and \( D_{n_i}(q_i) \).\qed
\medskip

We will apply Theorem~\ref{units} to the case when the ring in question acts on an abelian group by endomorphisms.

Let \( A \) be a finite abelian group given by its Cayley table. Let \( \End(A) \) be the ring of endomorphisms of \( A \).
We can encode every element \( f \in \End(A) \) as a tuple \( (f(a))_{a \in A} \) of \( |A| \) elements of \( A \),
so the additive group \( \End(A)^+ \) can be viewed as a subgroup of the direct power~\( A^{|A|} \).
The size of each element in this representation is polynomial in \( |A| \), and we can compute element sums and products (compositions) in polynomial time.
Note that in general the order of \( \End(A) \) is not polynomial in \( |A| \), but we have a bound \( |\End(A)| \leq |A|^{\log |A|} \).

Let \( \End(A)^+ \) denote the additive group of \( \End(A) \), and let \( e_1, \dots, e_t \) be cyclic generators of \( A \).
For each \( a \in A \) set \( N(a) = \{ b \in A \mid b^{|a|} = 1 \} \). For \( i = 1, \dots, t \) and \( a \in N(e_i) \)
we can define an endomorphism \( f_{i, a} \) of \( A \) by \( f_{i, a}(e_1^{x_1} \cdots e_t^{x_t}) = a^{x_i} \), 
where \( x_1, \dots x_t \in \mathbb{Z} \). Note that \( f_{i, a} \), \( i = 1, \dots, t \), \( a \in N(e_i) \), generate
\( \End(A)^+ \) additively.
One can compute the cyclic generators of \( \End(A)^+ \) and express any given element
in terms of these generators in time polynomial in~\( |A| \).

Let \( R \) be a unitary subring of \( \End(A) \) given by its additive generators.
Note that \( R^\times \) acts on \( A \) by automorphisms, in particular, every element of \( R^\times \) is a permutation on \( A \).
Hence we can specify \( R^\times \) as a permutation group in terms of its permutation generators; their number is polynomial in \( |A| \).
All prime divisors of \( |\End(A)| \) also divide \( |A| \), hence the largest prime divisor of \( |R| \) is bounded by~\( |A| \).

Theorem~\ref{units} implies the following.

\begin{corollary}\label{cunits}
	Let \( A \) be a finite abelian group given by its Cayley table. If \( R \) is a unitary subring of \( \End(A) \) given by its additive generators,
	then we can compute the permutation generators of \( R^\times \) in polynomial time.
\end{corollary}

\section{Proof of Theorem~\ref{main}}\label{sproof}

We separate the proof into two parts: finding a required isomorphism between \( G \) and \( G_0 \), and computing the coset of all isomorphisms.
Although we could combine both parts into one algorithm, it would be easier for understanding to consider these tasks separately.

\subsection{Finding an isomorphism}\label{findiso}

We will derive the first part of Theorem~\ref{main} from the following more precise statement.

\begin{proposition}\label{findis}
	Let \( G \) and \( G_0 \) be finite groups of order \( n \) given by their Cayley tables.
	Suppose we are also given abelian normal subgroups \( A \unlhd G \) and \( A_0 \unlhd G_0 \),
	and elements \( g_1, \dots, g_k \in G \), \( g^0_1, \dots, g^0_k \in G_0 \), such that
	\( g_1A, \dots, g_kA \) generate~\( G \) and \( g^0_1A, \dots, g^0_kA \) generate~\( G_0 \).
	We can test in time polynomial in \( n \) and \( k \) whether there exists an isomorphism \( \Phi : G \to G_0 \) such that \( \Phi(A) = A_0 \)
	and \( \Phi(g_i) = g^0_i \), \( i = 1, \dots, k \). If such an isomorphism exists, then we can find it in the same time.
\end{proposition}
\begin{proof}
	The algorithm is as follows.
	\begin{enumerate}
		\item Check that the map \( g_iA \mapsto g_i^0A_0 \), \( i = 1, \dots, k \), induces an isomorphism between \( G/A \) and \( G_0/A_0 \).
		\item Compute a presentation of \( H = G/A \) in terms of generators \( x_1, \dots, x_k \) and relators \( w_1, \dots, w_m \)
			such that \( g_iA \mapsto x_i \), \( i = 1, \dots, k \), induces an isomorphism between \( G/A \) and \( H \).
			Set \( u_j = w_j(g_1, \dots, g_k) \), \( j = 1, \dots, m \) and \( u^0_j = w_j(g^0_1, \dots, g^0_k) \), \( j = 1, \dots, m \).
		\item Find an isomorphism \( \mu : A \to A_0 \), such that \( \mu(x^{g_i}) = \mu(g)^{g^0_i} \) for all \( x \in A \) and \( i = 1, \dots, k \).
		\item Find an automorphism \( \theta : A \to A \), such that \( \theta(u_j) = \mu^{-1}(u^0_j) \)
			and \( \theta(x^{g_i}) = \theta(x)^{g_i} \) for all \( x \in A \) and \( i = 1, \dots, k \).
		\item Set \( \phi = \theta \circ \mu : A \to A_0 \),
			and compute \( \Phi \) using \( \Phi(w(g_1, \dots, g_k)x) = w(g^0_1, \dots, g^0_k) \phi(x) \), where \( x \in A \) and \( w \) is a word.
	\end{enumerate}

	Theorem~\ref{kgenext} implies that this algorithm will always correctly find an isomorphism if it exists.
	Now we explain how to perform each step in polynomial time.%
	\medskip

	\noindent
	\textbf{Step 1. Checking \( G/A \simeq G_0/A_0 \).}
	\medskip

	Since \( g_iA \), \( i = 1, \dots, k \), generate \( G/A \), we can express each element of \( G/A \)
	as a word in \( g_1A, \dots, g_kA \), and compute its image in \( G_0/A_0 \) via the map \( g_iA \mapsto g_i^0A_0 \).
	It is easy to check if the resulting map from \( G/A \) to \( G_0/A_0 \) is an isomorphism by bruteforce.
	\medskip

	\noindent
	\textbf{Step 2. Computing the presentation.}
	\medskip

	Given generators \( g_1A, \dots, g_kA \) of \( H = G/A \), we can compute the required presentation in terms of generators
	and relators by writing down the Cayley table of~\( H \). Indeed, for every element \( h \in H \) we can fix some word \( v_h(x_1, \dots, x_k) \)
	such that \( h = v_h(g_1A, \dots, g_kA) \); we can always assume that \( x_i = v_{x_i} \), \( i = 1, \dots, k \).
	The words \( v_{h_1}(\overline{x})v_{h_2}(\overline{x})v_{h_1h_2}(\overline{x})^{-1} \), \( h_1, h_2 \in H \),
	will serve as \( m = |H|^2 \) relators of our presentation. If we denote these words by \( w_1, \dots, w_m \),
	then we can easily compute \( u_j = w_j(\overline{g}) \) and \( u^0_j = w_j(\overline{g^0}) \), \( j = 1, \dots, m \).
	Note that the number of relators and their lengths are bounded polynomially in terms of \( k \) and~\( |H| \).
	\medskip

	\noindent
	\textbf{Step 3. Finding \( \mu \).}
	\medskip

	Recall that we have abelian subgroups \( A \unlhd G \) and \( A_0 \unlhd G_0 \)
	and elements \( g_1, \dots, g_k \in G \), \( g^0_1, \dots, g^0_k \in G_0 \).
	Moreover, \( g_1A, \dots, g_kA \) generate \( G/A \) and \( g^0_1A_0, \dots, g^0_kA_0 \) generate \( G_0/A_0 \), where \( H = G/A \simeq G_0/A_0 \).

	Set \( R = \mathbb{Z}_n[H] \), i.e.\ a group ring of \( H \) over the ring of integers modulo~\( n \).
	Note that the unitary ring \( R \) can be given by \( |H| \leq n \) additive generators.
	Now define the action of \( R \) on \( A \) by \( x^{x_i} = x^{g_i} \), \( i = 1, \dots, k \), \( x \in A \),
	and on \( A_0 \) by \( x_0^{x_i} = x_0^{g^0_i} \), \( i = 1, \dots, k \), \( x_0 \in A_0 \);
	the action of other elements of \( R \) can be deduced by linearity since \( A \) and \( A_0 \) are natural \( \mathbb{Z}_n \)-modules.

	By Proposition~\ref{modiso}, we can compute an \( R \)-isomorphism \( \mu : A \to A_0 \) in polynomial time.
	By the definition of the action of \( R \), we have \( \mu(x^{g_i}) = \mu(x^{x_i}) = \mu(x)^{x_i} = \mu(x)^{g^0_i} \), \( i = 1, \dots, k \), as required.
	\medskip

	\noindent
	\textbf{Step 4. Finding an automorphism \( \theta \).}
	\medskip

	Define a subring \( K \leq \End(A) \) by
	\[ K = \{ \theta \in \End(A) \mid \theta(x^{g_i}) = \theta(x)^{g_i} \text{ for all } x \in A,\, i = 1, \dots, k \}. \]
	Additive generators of \( K \) can be computed in polynomial time by solving a system of linear Diophantine equations,
	see Proposition~\ref{diosolve}. By Corollary~\ref{cunits}, we can compute multiplicative generators of the unit group
	\[ K^\times = \{ \theta \in \Aut(A) \mid \theta(x^{g_i}) = \theta(x)^{g_i} \text{ for all } x \in A,\, i = 1, \dots, k \} \]
	in polynomial time. \( K^\times \) is a permutation group on \( A \) and we can check if tuples
	\( (u_1, \dots, u_m) \) and \( (\mu^{-1}(u^0_1), \dots, \mu^{-1}(u^0_m)) \) lie in the same \( K^\times \)-orbit in
	polynomial time, and find the required \( \theta \in K^\times \) if they do.
	\medskip

	\noindent
	\textbf{Step 5: Computing \( \Phi \).}
	\medskip

	Since \( g_1A, \dots, g_kA \) generate \( G/A \), we can express each element \( g \in G \)
	in the form \( g = w(g_1, \dots, g_k)x \), where \( w \) is some word and \( x \in A \).
	Moreover, \( w \) and \( x \) can be found in polynomial time (although they are not unique).
	Given this decomposition of \( g \) we can compute its image under \( \Phi \) as \( \Phi(g) = w(g_1^0, \dots, g_k^0) \phi(x) \).
	It follows that we can compute \( \Phi \) in polynomial time, and it is easy to check if \( \Phi \) is an isomorphism by bruteforce.
\end{proof}

Now we derive the first part of Theorem~\ref{main} from the proposition above.
First we check that \( A \) and \( A_0 \) are isomorphic abelian groups, since there is no required isomorphism between \( G \) and \( G_0 \) if they are not.
Recall that \( G/A \) is \( k \)-generated, so we can find elements \( g_1, \dots, g_k \in G \) such that \( g_1A, \dots, g_kA \) generate \( G/A \) in
time polynomial in \( n^k \) by bruteforce. Next we bruteforce through all elements \( g^0_1, \dots, g^0_k \in G_0 \) such that
\( \psi(g_iA) = g^0_iA_0 \), \( i = 1, \dots, k \); note that there are at most \( |A|^k \) tuples of elements to consider.
For every such choice of \( g^0_1, \dots, g^0_k \) we apply Proposition~\ref{findis}. If the algorithm from Proposition~\ref{findis}
returns an isomorphism \( \Phi \) for at least one choice of elements \( g^0_1, \dots, g^0_k \), then we have found the required isomorphism.
It is clear that if \( \Phi \) exists, then the algorithm from Proposition~\ref{findis} will succeed for \( g^0_i = \Phi(g_i) \), \( i = 1, \dots, k \),
so our algorithm for the first part of Theorem~\ref{main} works correctly.

Note that if elements \( g_1, \dots, g_k \in G \) such that \( g_1A, \dots, g_kA \) generate \( G/A \) are given as a part of the input to the algorithm,
then we do not have to bruteforce for generators of \( G/A \) and the runtime of the algorithm can be bounded by a polynomial in \( n \) and \( |A|^k \).

\subsection{Computing the coset of isomorphisms}

As was mentioned in the introduction, to compute the coset of isomorphisms it suffices to
be able to compute \( \Aut_0(G, A) \), i.e.\ the subgroup of \( \Aut(G) \) which stabilizes \( A \) and acts trivially on \( G/A \).
We will reduce the problem of computing \( \Aut_0(G, A) \) to the following.

\begin{proposition}\label{gensfixed}
	Let \( G \) be a finite group given by its Cayley table.
	Suppose we are given a normal abelian subgroup \( A \) of \( G \), and elements \( g_1, \dots, g_k \in G \)
	such that \( g_1A, \dots, g_kA \) generate \( G/A \).
	Then we can find the group
	\[ C_{\overline{g}} = \{ \phi \in \Aut_0(G, A) \mid \phi(g_i) = g_i,\, i = 1, \dots, k \} \]
	in polynomial time.
\end{proposition}
\begin{proof}
	Note that \( C_{\overline{g}} \) is a permutation group on \( G \), so it suffices to compute its generators.

	As in Step~2 of Section~\ref{findiso}, compute a presentation of \( H = G/A \) in terms of generators \( x_1, \dots, x_k \) and relators \( w_1, \dots, w_m \)
	such that \( g_iA \mapsto x_i \), \( i = 1, \dots, k \), induces an isomorphism between \( G/A \) and \( H \).
	Set \( u_j = w_j(g_1, \dots, g_k) \), \( j = 1, \dots, m \). Define
	\[ K = \{ \theta \in \Aut(A) \mid \theta(x^{g_i}) = \theta(x)^{g_i} \text{ for all } x \in A,\, i = 1, \dots, k \}. \]
	As in Step~4 of Section~\ref{findiso}, we can compute permutation generators of \( K \) in polynomial time
	by solving a system of linear Diophantine equations and applying Corollary~\ref{cunits}.
	Let \( C \) be the pointwise stabilizer of \( u_1, \dots, u_m \) in \( K \), note that it can be computed in polynomial time.

	Notice that for any \( \phi \in C_{\overline{g}} \) the restriction of \( \phi \) on \( A \) lies in \( C \).
	In the other direction, for every \( \theta \in C \), Theorem~\ref{kgenext} implies that there exists a unique map \( \tilde{\theta} \in C_{\overline{g}} \)
	which induces \( \theta \) on \( A \) and can be defined by the formula \( \tilde{\theta}(w(g_1, \dots, g_k)x) = w(g_1, \dots, g_k)\theta(x) \)
	for all \( x \in A \) and any word \( w \).

	Let \( \theta_1, \dots, \theta_r \) be generators of \( C \).
	It follows from the paragraph above that \( \tilde{\theta_1}, \dots, \tilde{\theta_r} \) generate \( C_{\overline{g}} \),
	and since all the required computations can be performed in polynomial time, the proposition is proved.
\end{proof}

In order to find \( \Aut_0(G, A) \), fix some elements \( g_1, \dots, g_k \in G \) such that \( g_1A, \dots, g_kA \) generate \( G/A \).
For \( g^0_1, \dots, g^0_k \in G \) let \( \phi_{\overline{g}, \overline{g^0}} \)
denote an element \( \phi \) of \( \Aut(G) \) satisfying \( \phi(A) = A \) and \( \phi(g_i) = \phi(g^0_i) \), \( i = 1, \dots, k \), if it exists.
Now,
\[ \Aut_0(G, A) = \bigcup_{g^0_1, \dots, g^0_k \in G} C_{\overline{g}} \cdot \phi_{\overline{g}, \overline{g^0}}, \]
where the union is taken over all elements \( g^0_1, \dots, g^0_k \in G \) such that \( g^0_i \in g_iA \), \( i = 1, \dots, k \),
and \( \phi_{\overline{g}, \overline{g^0}} \) exists. Note that there are polynomially many cosets in this union (at most \( |A|^k \)).
We can compute \( C_{\overline{g}} \) in polynomial time by Proposition~\ref{gensfixed}, and we can find representatives
\( \phi_{\overline{g}, \overline{g^0}} \) in polynomial time by Proposition~\ref{findis}.
It follows that we can compute \( \Aut_0(G, A) \) in polynomial time, Theorem~\ref{main} is proved.

Note that as in Section~\ref{findiso}, if elements \( g_1, \dots, g_k \in G \) such that \( g_1A, \dots, g_kA \) generate \( G/A \)
are given as a part of the input to the algorithm, then the runtime of the algorithm can be bounded by a polynomial in \( n \) and \( |A|^k \).

\section{Proof of Theorem~\ref{ctower}}\label{stower}

\begin{proposition}\label{findcf}
	Let \( G \) be a finite group given by its Cayley table.
	One can find all normal subgroups \( A \) of \( G \) such that \( G/A \) is cyclic or simple in polynomial time.
\end{proposition}
\begin{proof}
	To enumerate simple factors one can use the algorithm in \cite[Lemma~7.4]{babainc},
	see also~\cite[Section~3.1~(o) and (p)]{seress}, so it suffices to find all \( A \)
	such that \( G/A \) is cyclic. Since every such \( A \) contains \( [G, G] \), we may replace \( G \) by \( G/[G, G] \) and assume that \( G \) is abelian.

	Let \( e \) be the exponent of \( G \); note that it suffices to compute all homomorphisms from \( G \) to \( \mathbb{Z}_e \).
	Decompose \( G \) into a direct product of cyclic groups \( C_1, \dots, C_k \), and compute all homomorphisms
	from \( C_i \) to \( \mathbb{Z}_e \), \( i = 1, \dots, k \). Every such choice defines a unique homomorphism from \( G \) to \( \mathbb{Z}_e \),
	and every required homomorphism can be obtained that way.
	There are exactly \( |G| \) such homomorphisms, so the algorithm works in polynomial time.
\end{proof}
\begin{corollary}\label{findall}
	Let \( G \) be a finite group of order \( n \) given by its Cayley table and let \( k \geq 1 \) be an integer.
	In time polynomial in \( n^k \) one can find all normal subgroups \( A \) of \( G \) such that \( G/A \) has a subnormal series of length \( k \) with
	cyclic or simple factors.
\end{corollary}
\begin{proof}
	By applying Proposition~\ref{findcf} recursively \( k \) times, in time polynomial in \( n^k \)
	we can enumerate all subnormal series \( N_k \unlhd N_{k-1} \unlhd \dots \unlhd N_0 = G \) where \( N_i/N_{i+1} \), \( i = 0, \dots, k-1 \), is cyclic or simple. 
	One can easily check if \( N_k \) is abelian and normal in \( G \) in polynomial time, so the claim follows.
\end{proof}

\noindent
\emph{Proof of Theorem~\ref{ctower}.}
By Corollary~\ref{findall}, we can find a normal abelian subgroup \( A \) of \( G \) such that
\( G/A \) has a subnormal series of length \( k \) with cyclic or simple factors. Similarly, we can list all abelian normal subgroups \( A_0 \)
of \( G_0 \), such that \( G_0/A_0 \) has a subnormal series of length \( k \) with cyclic or simple factors.
Note that \( G/A \) and \( G_0/A_0 \) are \( 2k \)-generated.

For each choice of \( A_0 \) we list all isomorphisms \( \psi : G/A \to G_0/A_0 \) (if they exist) by bruteforce in time polynomial in \( n^{2k} \)
and apply Theorem~\ref{main} to find the coset of isomorphisms. The union of at most \( n^{2k} \) of these cosets will be the set of all
isomorphisms from \( G \) to \( G_0 \). In order to express this union of cosets as a \( \Aut(G) \)-coset,
it suffices to compute \( \Aut(G) \) as a permutation group on \( G \).

To do that, we apply the above procedure to \( G_0 = G \) and obtain \( \Aut(G) \) as a union of cosets
\( H_1\phi_1 \cup \dots \cup H_r\phi_r \), where \( H_i \leq \Aut(G) \), \( \phi_i \in \Aut(G) \), \( i = 1, \dots, r \), \( r \leq n^{2k} \).
Now the union of generators of \( H_1, \dots, H_r \) together with \( \phi_1, \dots, \phi_r \) generate \( \Aut(G) \) as a permutation group.
\qed

\section{Acknowledgements}

The author expresses his gratitude to A.V.~Vasil'ev and I.N.~Ponomarenko for fruitful discussions on the topic of this paper and many useful suggestions.

The work is supported by the Russian Science Foundation, project 24-11-00127, \url{https://rscf.ru/en/project/24-11-00127/}.

\bigskip

\noindent
\emph{Saveliy V. Skresanov}

\noindent
\emph{Novosibirsk State University, 1 Pirogova St.,}

\noindent
\emph{Sobolev Institute of Mathematics, 4 Acad. Koptyug avenue,\\ 630090 Novosibirsk, Russia}

\noindent
\emph{Email address: skresan@math.nsc.ru}

\end{document}